\documentclass[11pt,reqno]{amsart}

\usepackage[driverfallback=hypertex]{hyperref}
\usepackage{graphicx}
\usepackage{color}
\usepackage{amsfonts,amssymb}
\usepackage{bbm} 

\usepackage{mathrsfs}

\usepackage{pdfsync}

\usepackage{a4wide}

\newtheorem{neu}{}[section]
\newtheorem{Cor}[neu]{Corollary}
\newtheorem*{Cor*}{Corollary}
\newtheorem{Thm}[neu]{Theorem}
\newtheorem*{Thm*}{Theorem}
\newtheorem*{Observation*}{Observation}
\newtheorem{Prop}[neu]{Proposition}
\newtheorem*{Prop*}{Proposition}
\theoremstyle{definition}
\newtheorem{Lemma}[neu]{Lemma}
\newtheorem*{Rmk*}{Remark}
\newtheorem{Rmk}[neu]{Remark}

\newtheorem*{Ex*}{Example}

\newtheorem*{Qu*}{Question}

\newtheorem{Def}[neu]{Definition}

\newcommand{\R}{\mathbb{R}}
\newcommand{\C}{\mathbb{C}}
\newcommand{\om}{\omega}
\renewcommand{\L}{\mathscr{L}}

\newcommand{\beq}{\begin{equation}}
\newcommand{\beqn}{\begin{equation}\nonumber}
\newcommand{\eeq}{\end{equation}}

\newcommand{\bea}{\begin{equation}\begin{aligned}}
\newcommand{\bean}{\begin{equation}\begin{aligned}\nonumber}
\newcommand{\eea}{\end{aligned}\end{equation}}

\numberwithin{equation}{section}

\definecolor{Urs}{rgb}{0,.7,0}
\definecolor{Peter}{rgb}{0,0,1}
\definecolor{red}{rgb}{1,0,0}

\DeclareMathOperator{\id}{{Id}}
\DeclareMathOperator{\im}{Im}
\DeclareMathOperator{\re}{Re}
\DeclareMathOperator{\si}{sign}
\DeclareMathOperator{\pa}{par}

\begin{document}
\title{A $\Gamma$-structure on Lagrangian Grassmannians}
\author{Peter Albers}
\author{Urs Frauenfelder}
\author{Jake P. Solomon}
\address{
    Peter Albers\\
    Mathematisches Institut\\
    Westf\"alische Wilhelms-Universit\"at M\"unster}
\email{peter.albers@wwu.de}
\address{
    Urs Frauenfelder\\
    Department of Mathematics and Research Institute of Mathematics\\
    Seoul National University}
\email{frauenf@snu.ac.kr}
\address{
    Jake P. Solomon\\
    Institute of Mathematics\\
    Hebrew University of Jerusalem}
\email{jake@math.huji.ac.il}
\keywords{}
\begin{abstract}
For $n$ odd the Lagrangian Grassmannian of $\R^{2n}$ is a $\Gamma$-manifold.
\end{abstract}
\maketitle

\section{Introduction and statement of the result}

We denote by $(\R^{2n},\om)$ the standard symplectic vector space. The (unoriented) Lagrangian Grassmannian $\L$ is the space of all Lagrangian subspaces of $\R^{2n}$. It is a homogeneous space
\[
\L\cong U(n)/O(n),
\]
see \cite{Arnold_Givental_Symplectic_geometry,McDuff_Salamon_introduction_symplectic_topology}. Every Lagrangian subspace can be identified with the fixed point set of a linear orthogonal anti-symplectic involution. Using this identification, we define a smooth map
\[
\Theta:\L\times\L\to\L
\]
by
\[
(R,S)\mapsto RSR,
\]
which we think of as a product. On every space there are products such as  constant maps and projections to one factor. In \cite{Hopf_Uber_die_Topologie_der_Gruppenmanigfaltigkeiten} Hopf introduced the notion of $\Gamma$-manifolds which rules these trivial products out. The purpose of this paper is to prove that the above product gives the Lagrangian Grassmannian $\L$ the structure of a $\Gamma$-manifold for $n$ odd.

\begin{Def}
A closed, connected, orientable manifold $M$ carries the structure of a $\Gamma$-manifold if there exists a map
\[
\Theta:M\times M\to M
\]
such that the maps
\[
x\mapsto\Theta(x,y_0)\quad\text{and}\quad y\mapsto\Theta(x_0,y)
\]
have non-zero mapping degree for one and thus all pairs $(x_0,y_0)\in M\times M$.
\end{Def}

It is well-known that $\L$ is orientable if and only if $n$ odd, see \cite{Fuks_Maslov_Arnold_characteristic_classes}. The main result of this article is the following theorem.

\begin{Thm}\label{thm:thm_1}
If $n$ is odd, then $(\L,\Theta)$ is a $\Gamma$-manifold.
\end{Thm}

Using Hopf's theorem \cite[Satz 1]{Hopf_Uber_die_Topologie_der_Gruppenmanigfaltigkeiten}, we get a new proof of the following Corollary due to Fuks \cite{Fuks_Maslov_Arnold_characteristic_classes}.
\begin{Cor}[\cite{Fuks_Maslov_Arnold_characteristic_classes}]
For $n$ odd, the rational cohomology ring of $\L$ is an exterior algebra on generators of odd degree.
\end{Cor}

\begin{Rmk}
The cohomology ring of the oriented and unoriented Lagrangian Grassmannian was computed by Borel and Fuks for all $n$, see
\cite{Borel_La_cohomologie_mod_2,Borel_Sur_la_cohomologie_des_espaces_fibres_principaux,Fuks_Maslov_Arnold_characteristic_classes}. A nice summary of these results can be found in the book of Vassilyev \cite[Chapter 22]{Vassilyev_Lagrange_and_Legendre_characteristic_classes}.
\end{Rmk}

The above situation fits into the following general framework. It is well-known that $\L$ embeds into $U(n)$ as the set $U(n)\cap Sym(n)$, i.e.~the symmetric unitary matrices. Thus, $\L$ can be interpreted as the fixed point set of the involutive anti-isomorphism $A\mapsto A^T$ of $U(n)$. On any Lie group $G$ we can define a new product: $(g,h)\mapsto gh^{-1}g$. If $I:G\to G$ is an involutive anti-isomorphism then this new product restricts to a product on the fixed point set $\mathrm{Fix}(I)$. This is precisely the situation for the Lagrangian Grassmanian, namely the  map $\Theta$ corresponds under the embedding of $\L$ into $U(n)$ to $(g,h)\mapsto gh^{-1}g$.

For general Lie groups this new product does not always give rise to a $\Gamma$-structure for various reasons. For example, if we take $G=O(n)$ resp.~$G=U(n)$ and $I(A):=A^{-1}$, then $\mathrm{Fix}(I)$ can be identified with $\cup_k G(k,n)$, the union of all real resp.~complex Grassmannians, which is not connected. Another example is $G=SU(n)$ with $I=$ transposition. Then for $n=2$ we can identify $\mathrm{Fix}(I)\cong S^2$. But by Hopf's theorem \cite[Satz 1]{Hopf_Uber_die_Topologie_der_Gruppenmanigfaltigkeiten} $S^2$ is not a $\Gamma$-manifold.

\subsection*{Acknowledgements}
Parts of this article were written during a visit of the first two authors at the Forschungs\-institut f\"ur Mathematik (FIM), ETH Z\"urich. The authors thank the FIM for its stimulating working atmosphere.

This material is supported by the SFB 878 -- Groups, Geometry and Actions (PA), by the Alexander von Humboldt Foundation (UF), by Israel Science Foundation grant 1321/2009 and by Marie Curie grant No. 239381 (JPS).

\section{Proof of Theorem \ref{thm:thm_1}}
We recall that the (unoriented) Lagrangian Grassmannian $\L$ is the homogeneous space
\[
\L\cong U(n)/O(n),
\]
see \cite{Arnold_Givental_Symplectic_geometry,McDuff_Salamon_introduction_symplectic_topology}. Since $n$ is odd, $\L$ is a closed connected orientable manifold \cite{Fuks_Maslov_Arnold_characteristic_classes}. The space $\L$ is naturally identified with the space of linear orthogonal anti-symplectic involutions of $\R^{2n}$ with the standard symplectic structure.
Using this identification, we define the map
\[
\Theta:\L\times\L\to\L
\]
by $(R,S)\mapsto RSR$.
In order to prove Theorem \ref{thm:thm_1}, it suffices to show for one choice of basepoint $R_0$ that the mapping degrees of
\[
S\mapsto\Theta(R_0,S)\quad\text{and}\quad S\mapsto\Theta(S,R_0)
\]
are non-zero. Since
\[
S\mapsto \Theta(R_0,S)=R_0SR_0\mapsto \Theta(R_0,R_0SR_0)=R_0R_0SR_0R_0=S,
\]
the first map is an involution and therefore has mapping degree $\pm1$. The non-trivial case is to compute the mapping degree of
\[
\Theta_0(S):=\Theta(S,R_0)=SR_0S\;.
\]
Theorem \ref{thm:thm_1} follows immediately from the following proposition.
\begin{Prop}\label{prop:degree}
The mapping degree of $\Theta_0$ equals
\[
\deg \Theta_0 =2^{m+1}
\]
where $n=2m+1$.
\end{Prop}
\begin{proof}
Identify $\R^{2n} = \C^n$ in the standard way. Denote by $\tau : \C^n \to \C^n$ the map given by complex conjugation of all coordinates simultaneously. It is a standard fact, see for instance \cite{McDuff_Salamon_introduction_symplectic_topology}, that an orthogonal symplectic map $\R^{2n} \to \R^{2n}$ corresponds to a unitary map $\C^n \to \C^n.$ It follows that an orthogonal anti-symplectic map $R : \R^{2n} \to \R^{2n}$ can be written as the composition $A \circ \tau : \C^n \to \C^n$ for $A$ a unitary linear map. The condition $R^2 = \id$ then translates to $A\overline A = \id$. So, we identify
\[
\L = \{ A \in U(n)| A \overline A = \id\}.
\]
Under this identification, the map $\Theta$ is given by 
\[
\Theta(A,B) = A\overline B A.
\]
Let $B_0$ be the unitary matrix corresponding to $R_0.$ Then the map $\Theta_0$ is given by 
\[
\Theta_0(A) = \Theta(A,B_0) = A\overline B_0 A.
\]
In the following, we take $B_0 = B,$ the diagonal matrix with entries $b_{jk} = e^{i\theta_j}\delta_{jk}$ where
\[
0 < \theta_j < 2\pi, \qquad \theta_1 < \theta_2 < \cdots < \theta_n.
\]
For this choice of $B_0,$ we show that $\id$ is a regular value of $\Theta_0$ and compute the signed cardinality of $\Theta_0^{-1}(\id).$

Indeed, if $\Theta_0(A) = \id,$ then $A \overline B A = \id,$ and therefore $\overline A B = A.$ Throughout this paper, we do \emph{not} use the Einstein summation convention. Letting $a_{jk}$ denote the matrix entries of $A,$ we have
\[
\bar a_{jk} e^{i\theta_k} = a_{jk}.
\]
Write $a_{jk} = r_{jk} e^{i\psi_{jk}},$ where $r_{jk} \in \R$ and $0 \leq \psi_{jk} < \pi.$ So,
\[
e^{i2\psi_{jk}} = a_{jk}/\bar a_{jk} = e^{i\theta_k},
\]
and therefore $\psi_{jk} = \theta_k/2.$ Writing the unitary condition for $A$ in terms of $r_{jk}$ and $\psi_{jk},$ we have
\[
\delta_{jl} = \sum_k a_{jk} \bar a_{lk} = \sum_k r_{jk}r_{lk} e^{i(\psi_{jk} -\psi_{lk})} = \sum_k r_{jk}r_{lk}.
\]
Thus $r_{jk}$ is an orthogonal matrix. Furthermore, the condition $A\overline A = \id$ translates to
\[
\delta_{jl} = \sum_k a_{jk} \bar a_{kl} = \sum_k r_{jk}r_{kl} e^{i(\theta_k - \theta_l)/2}.
\]
In particular, taking $j = l,$ we obtain
\[
1 = \sum_k r_{jk}r_{kj} \cos((\theta_k - \theta_j)/2).
\]
Writing $r_{jk}' = \cos((\theta_k - \theta_j)/2)r_{jk},$ we can reformulate the preceding equation in terms of the inner product of the row and column vectors $r_{j\cdot}'$ and $r_{\cdot j}.$ Namely,
\begin{equation}\label{eq:ip}
r_{j\cdot}' \cdot r_{\cdot j}=1\;.
\end{equation}
On the other hand, since $r_{jk}$ is unitary, $|r_{\cdot j}| = 1$ and
\[
|r_{j\cdot}'|^2 = \sum_k r_{jk}^2 \cos^2((\theta_k - \theta_j)/2) \leq \sum_k r_{jk}^2 = |r_{j\cdot}|^2 = 1,
\]
with equality only if $r_{jk} = 0$ when $k \neq j.$ Applying Cauchy-Schwartz to equation~\eqref{eq:ip}, we have
\[
1 \leq |r_{j\cdot}'||r_{\cdot j}| = |r_{j\cdot}'| \leq 1.
\]
Thus equality must hold, and the matrix $r_{jk}$ is diagonal. Moreover, orthogonality implies that $r_{jk} = \pm \delta_{jk}.$ Summing up, $A \in \Theta_0^{-1}(\id)$ if and only if we have $A = A^\epsilon,$ where
\[
\epsilon = (\epsilon_1,\ldots,\epsilon_n), \qquad \epsilon_k \in \{0,1\},
\]
and $A^\epsilon$ is the matrix with elements
\[
a_{jk}^\epsilon =  e^{i(\theta_k/2 + \epsilon_k \pi)}.
\]
In particular, $\Theta_0^{-1}(\id)$ has unsigned cardinality $2^n.$

It remains to show that $\id$ is a regular value and compute the signs. Let $Sym(n)$ denote the space of real $n\times n$ symmetric matrices. It is easy to see that the tangent space to $\L$ at $A = \id$ is given by
\[
T_{\id} \L = \{T \in \mathfrak u(n)|T + \overline T = 0\} = \{iQ | Q \in Sym(n)\}.
\]
Recall that $U(n)$ acts on $\L$ by $A \mapsto U A \overline U^{-1}.$ Thus, if $A = U \overline U^{-1},$ we have an isomorphism
\[
\kappa_U : T_{\id} L \overset{\sim}{\longrightarrow} T_A \L
\]
given by $T \mapsto U T \overline U^{-1}.$ Since $\L$ is a $U(n)$ homogeneous space, the isomorphism $\kappa_U$ preserves orientation. Moreover, for $T \in T_{A^\epsilon} \L$ we have
\[
d\Theta_0|_{A^\epsilon}(T) = T \overline B A^\epsilon + A^\epsilon \overline B T = T \overline A^\epsilon + \overline A^\epsilon T.
\]
If $U^\epsilon \in U(n)$ satisfies
\[
A^\epsilon = U^\epsilon (\overline U^\epsilon)^{-1},
\]
then $A^\epsilon$ is a regular point of $\Theta_0$ if the linear map
\[
\alpha^\epsilon = d\Theta_0|_{A^\epsilon} \circ \kappa_U : T_{\id} \L \to T_{\id} \L
\]
is invertible, and in that case the sign of $A^\epsilon$ is $\si \det(\alpha^\epsilon).$
Explicitly,
\begin{align*}
\alpha^\epsilon(T) &= U^\epsilon T (\overline U^\epsilon)^{-1} \overline A^\epsilon + \overline A^\epsilon U^\epsilon T (\overline U^\epsilon)^{-1} \\
&= U^\epsilon T (U^\epsilon)^{-1} + \overline U^\epsilon T (\overline U^\epsilon)^{-1} \\
&= U^\epsilon T (U^\epsilon)^{-1} - \overline U^\epsilon \overline T (\overline U^\epsilon)^{-1} \\
&= 2i\im(U^\epsilon T (U^\epsilon)^{-1}).
\end{align*}
Writing $T = iQ,$ we can think of $\alpha^\epsilon$ as the map $Sym(n) \to Sym(n)$ given by
\[
\alpha^\epsilon(Q) = 2\re(U^\epsilon Q (U^\epsilon)^{-1}).
\]
For convenience, we take $U^\epsilon$ to be the unitary linear map given by
\[
u_{jk}^\epsilon = e^{i(\theta_k/4 + \epsilon_k \pi/2)}\delta_{jk}.
\]
Then, denoting by $q_{jk}$ the matrix elements of $Q,$ we have
\begin{align*}
\alpha^\epsilon(Q)_{jk} &= 2\re\left(e^{i\big((\theta_j-\theta_k)/4 + (\epsilon_j -\epsilon_k)\pi/2\big)}\right)q_{jk} \\
 & = 2 \cos\big((\theta_j-\theta_k)/4 + (\epsilon_j -\epsilon_k)\pi/2\big) q_{jk}.
\end{align*}
Since $Q$ is a symmetric matrix, it is determined by $q_{jk}$ for $j\leq k.$ Thus
\[
\det(\alpha^\epsilon) = \prod_{j\leq k} 2 \cos\big((\theta_j-\theta_k)/4 + (\epsilon_j -\epsilon_k)\pi/2\big) q_{jk}.
\]
We need to show this determinant does not vanish and compute its sign.
For $j = k,$ clearly $\cos\big((\theta_j-\theta_k)/4 + (\epsilon_j -\epsilon_k)\pi/2\big) = 1$.
For $j < k,$ by assumption, $0 < \theta_j < \theta_k < 2\pi,$ so
\[
-\frac{\pi}{2} < \frac{\theta_j - \theta_k}{4} < 0.
\]
It follows that for all $j\leq k$, we have $\cos\big((\theta_j-\theta_k)/4 + (\epsilon_j -\epsilon_k)\pi/2\big) \neq 0.$ Therefore $\det(\alpha^\epsilon) \neq 0$ for all $\epsilon$ and $\id$ is a regular value. Moreover,
\[
\cos\big((\theta_j-\theta_k)/4 + (\epsilon_j -\epsilon_k)\pi/2\big) < 0 \quad \Longleftrightarrow \quad \epsilon_j = 0,\, \epsilon_k = 1.
\]
Let $\Upsilon_n$ be the set of all binary sequences $\epsilon = (\epsilon_1,\ldots,\epsilon_n).$
For $\epsilon \in \Upsilon_n$ define $\si(\epsilon)$ to be the number modulo $2$ of pairs $j<k$ such that $\epsilon_j =0$ and $\epsilon_k = 1.$ The upshot of the preceding calculations is that
\[
\si \det (\alpha^\epsilon) = \si(\epsilon),
\]
therefore
\[
\deg \Theta_0 = \sum_{\epsilon \in \Upsilon_n} (-1)^{\si(\epsilon)}.
\]
A combinatorial argument given below in Lemma~\ref{lm:comb} then implies the theorem.
\end{proof}

\begin{Lemma}\label{lm:comb}
For $n = 2m+1,$ we have
\[
d_n := \sum_{\epsilon \in \Upsilon_n} (-1)^{\si(\epsilon)} = 2^{m+1}.
\]
\end{Lemma}
\begin{proof}
Let $M_n$ denote the number of $\epsilon \in \Upsilon_n$ such that $\si(\epsilon) = 0.$ Then
\[
d_n = M_n - (2^n - M_n) = 2 M_n - 2^n.
\]
For $\epsilon \in \Upsilon_n$ denote by $\pa(\epsilon)$ the parity of $\epsilon,$ or in other words the number modulo $2$ of $j$ such that $\epsilon_j = 1.$ Let $P_n$ denote the number of $\epsilon \in \Upsilon_n$ such that $\si(\epsilon) + \pa(\epsilon) = 0.$ By analyzing what happens when we adjoin either $1$ or $0$ to the beginning of a sequence $\epsilon \in \Upsilon_{n-1},$ we find that
\[
M_n = M_{n-1} + P_{n-1}, \qquad P_n = (2^{n-1} - P_{n-1}) + M_{n-1}.
\]
Iterating these recursions twice, we obtain
\[
M_n = M_{n-2} + P_{n-2} + 2^{n-2} - P_{n-2} + M_{n-2} = 2M_{n-2} + 2^{n-2}.
\]
Clearly $M_1 = 2,$ so $d_1 = 2.$ Using the preceding recursion for $M_n,$ we obtain
\[
d_n = 2(2M_{n-2} + 2^{n-2}) - 2^n = 2(2M_{n-2} - 2^{n-2}) = 2 d_{n-2}.
\]
The lemma follows by induction.
\end{proof}

%
\bibliographystyle{amsalpha}
\bibliography{../../../../Bibtex/bibtex_paper_list}
\end{document}